\documentclass{amsart}
\usepackage{dsfont, amssymb, mathtools, IEEEtrantools, subcaption} %, dutchcal}
\usepackage{tabularray} %Customizable tables with fixed width columns
\usepackage{calc,changepage}
\usepackage{csquotes}
\usepackage{enumitem}
\usepackage{booktabs}
\usepackage{diagbox}

\usepackage{xcolor}
\definecolor{wb}{RGB}{51,153,255}
\usepackage{tikz}
\usetikzlibrary{calc}
\usepackage{tikz-cd}

\usepackage[parfill]{parskip}
\numberwithin{equation}{subsection}
\usepackage{amsrefs}

\newcommand{\defeq}{\vcentcolon=}
\newcommand{\eqdef}{=\vcentcolon}

\makeatletter
\def\moverlay{\mathpalette\mov@rlay}
\def\mov@rlay#1#2{\leavevmode\vtop{%
   \baselineskip\z@skip \lineskiplimit-\maxdimen
   \ialign{\hfil$\m@th#1##$\hfil\cr#2\crcr}}}
\newcommand{\charfusion}[3][\mathord]{
    #1{\ifx#1\mathop\vphantom{#2}\fi
        \mathpalette\mov@rlay{#2\cr#3}
      }
    \ifx#1\mathop\expandafter\displaylimits\fi}
\makeatother

\newcommand{\cupdot}{\charfusion[\mathbin]{\cup}{\cdot}}
\newcommand{\bigcupdot}{\charfusion[\mathop]{\bigcup}{\cdot}}

\DeclareFontFamily{U}{mathb}{\hyphenchar\font45}
\DeclareFontShape{U}{mathb}{m}{n}{
      <5> <6> <7> <8> <9> <10> gen * mathb
      <10.95> mathb10 <12> <14.4> <17.28> <20.74> <24.88> mathb12
      }{}
\DeclareSymbolFont{mathb}{U}{mathb}{m}{n}
\DeclareMathSymbol{\precneq}{3}{mathb}{"AC}
\DeclareMathSymbol{\varprec}{3}{mathb}{"A0}

\newtheoremstyle{definitions}
 	{\topsep}% measure of space to leave above the theorem. E.g.: 3pt
	{\topsep}% measure of space to leave below the theorem. E.g.: 3pt
	{}% name of font to use in the body of the theorem
	{}% measure of space to indent
	{\bfseries}% name of head font
	{:}% punctuation between head and body
	{.5em}% space after theorem head; " " = normal interword space
	{}
  
\newtheoremstyle{lemmata}
	{\topsep}% measure of space to leave above the theorem. E.g.: 3pt
	{\topsep}% measure of space to leave below the theorem. E.g.: 3pt
	{\itshape} %{\slshape}% name of font to use in the body of the theorem
	{}% measure of space to indent
	{\bfseries}% name of head font
	{:}% punctuation between head and body
	{.5em}% space after theorem head; " " = normal interword space
	{}

\theoremstyle{lemmata}
\newtheorem{Theorem}[subsection]{Theorem}

\newtheorem{Corollary}[subsection]{Corollary}
\newtheorem{Proposition}[subsection]{Proposition}

\theoremstyle{definitions}

\newtheorem{Remark}[subsection]{Remark}
\newtheorem{Remarks}[subsection]{Remarks}
\newtheorem*{Remarks-nn}{Remarks}

\DeclareMathOperator{\GL}{GL}

\DeclareMathOperator{\aut}{aut}

\DeclareMathOperator{\sgn}{sgn}
\DeclareMathOperator{\supp}{supp}

\usepackage{hyperref}
\allowdisplaybreaks

\title{Modular forms for \(\GL(r, \mathds{F}_{q}[T])\): \(t\)-expansions of the basic forms}
\author{Ernst-Ulrich Gekeler \\ Universität des Saarlandes \\ \lowercase{\href{mailto:gekeler@math.uni-sb.de}{gekeler@math.uni-sb.de}}}
\address{FR 6.1 Mathematik, Universität des Saarlandes, Postfach 15 11 50 D-66041 Saarbrücken}
\email{gekeler@math.uni-sb.de}
\date{\today}
\subjclass{MSC Primary 11F52; Secondary 11G09}
\keywords{Drinfeld modular forms, Eisenstein series, coefficient forms, \(t\)-expansions}

\begin{document}

\begin{abstract}
	We give closed formulas for the first few expansion coefficients of the basic modular forms for \(\GL(r, \mathds{F}_{q}[T])\). Here the rank \(r\)
	is larger or equal to \(3\), and the forms in question include the coefficient forms \(g_{1}, \dots, g_{r}\) and the Eisenstein series
	\(E_{q^{i}-1}\) (\(i \in \mathds{N}\)).	
\end{abstract}

\maketitle

\setcounter{section}{-1}

\section{Introduction}

In recent years, Drinfeld modular forms for groups like \(\Gamma \defeq \GL(r, \mathds{F}_{q}[T])\) have been introduced and investigated in more 
detail in \cite{BassonBreuerPink24} and in the author's series of papers starting with \cite{Gekeler17}, \cite{Gekeler22}, \cite{Gekeler22-2}, see also \cite{Gekeler25-2}. As in the case of rank \(r=2\), where a large
amount of literature exists, it is desirable to dispose of examples and numerical data allowing a deeper arithmetical study. For that case, see
e.g. the works of Bosser \cite{Bosser02}, Choi \cite{Choi08}, \cite{ChoiIm14}, Lopez \cite{Lopez10}, Vincent \cite{Vicent10}, Armana \cite{Armana11} and, more recently, of Bandini-Valentino \cite{BandiniValentino23} and Dalal-Kumar \cite{DalalKumar23}.
With the present paper, the author aims to provide some information about the \(t\)-expansions of the most basic modular forms for \(\Gamma\), namely
the Eisenstein series \(E_{q^{k}-1}\) and the coefficient forms \(g_{k}\), both of weight \(q^{k}-1\) and type \(0\). 

Each modular form for \(\Gamma\) has a \(t\)-expansion
\[
	f(\boldsymbol{\omega}) = \sum_{n \geq 0} a_{n}(\boldsymbol{\omega}')t^{n}(\boldsymbol{\omega}),
\]
where \(a_{n}(\boldsymbol{\omega}')\) is a weak modular form of rank \(r' \defeq r-1\), and \(\boldsymbol{\omega}'=(\omega_{2}, \dots, \omega_{r})\)
is the tail of \(\boldsymbol{\omega} = (\omega_{1}, \dots, \omega_{r})\), a point of the Drinfeld space \(\Omega^{r}\). We are able to give closed
formulas for the \(a_{n}\) if 
\begin{enumerate}[label=(\alph*)]
	\item \(f = E_{q^{k}-1}\) and \(n < 3(q-1)q^{k+r-2}\), see Theorem \ref{Theorem.Expansion-for-Eqk-1};
	\item \(f = g_{k}\) and \(n < 3(q-1)q^{r-1}\) (Theorem \ref{Theorem.gk-as-power-series-in-t});
\end{enumerate}
and, in a less satisfactory manner, if 
\begin{enumerate}[label=(\alph*)]\setcounter{enumi}{2}
	\item \(f = h\), the \((q-1)\)-th root of the discriminant form \(\Delta = g_{r}\), and \(n \leq (q-1)(q^{2r-1}-q^{r})\) (Theorem \ref{Theorem.h-function-as-power-series-in-t}).
\end{enumerate}
For (a) and (b) we have to assume \(q>2\), otherwise (as \(\sum_{c \in \mathds{F}_{q}} c \neq 0\) for \(q=2\), in contrast with \(q>2\)) some annoying
extra terms appear that lead to weaker error estimates.

In case (a), there are precisely \(k+1\) non-vanishing coefficients \(a_{n}\) below the given bound with \(n = q^{k}-q^{i}\) (\(0\leq i \leq k\));
in case (b), there are precisely \(2k\) (if \(k<r\)) or \(2r-1\) (if \(k=r\)) non-vanishing \(a_{n}\) below the bound, all with \(n \leq q^{k}-1\);
in case (c), we have no similar description. 

For (a), we start with the known \(A\)-expansion involving
\[
	\sum_{a \in A \text{ monic}} G_{q^{k}-1}(t_{a})
\]
of Theorem \ref{Theorem.Eisenstein-series-Eqk-1-t-expansion}. Due to the structure of the quantities \(t_{a}\) and of the Goss polynomials 
\(G_{q^{k}-1}\), there is heavy cancellation in the partial sums
\[
	\sum_{a \text{ monic of degree } d} G_{q^{k}-1}(t_{a}),
\]
which leads to strong estimates and thereby to Theorem \ref{Theorem.Expansion-for-Eqk-1}. For (b) we could try to use (a) and the well-known
relationship \eqref{Eq.Logarithm-and-Drinfeld-module} between the \(E_{q^{i}-1}\) and the \(g_{j}\). This, however, turns out hopelessly 
complicated. Instead, we use the presentation of the generic Drinfeld modulue \(\phi = \phi^{\boldsymbol{\omega}}\) of rank \(r\) (where
\(\boldsymbol{\omega}\) varies through the Drinfeld symmetric space \(\Omega^{r}\)) as a Tate object over \(\phi' = \phi^{\boldsymbol{\omega}'}\),
its counterpart of rank \(r' = r-1\). This leads to the more efficient recursion formula \eqref{Eq.Recursion-formula-gk} for the \(g_{k}\), which
is better adapted to our problem. Together with Proposition \ref{Proposition.E-qi-1-in-R} (a technical estimate for the Tate object), it allows
to show the formula \eqref{Eq.Closed-form-gk} for the non-vanishing coefficients, that is, Theorem \ref{Theorem.gk-as-power-series-in-t}.

The situation is different for (c), i.e., for the modular form \(h\) of weight \((q^{r}-1)/(q-1)\) and type 1. In contrast with the Eisenstein
series \(E_{q^{k}-1}\) and the coefficient forms \(g_{k}\) (whose \(t\)-expansions are sparse), but in accordance with the case of rank \(2\)
(see \cite{Gekeler99}), \(h\) has \enquote{many} non-vanishing coefficients. We restrict to estimating the contributions of monic \(a \in A\) of degree
\(\geq 2\) in the product formula
\[
	h = (h')^{q} t \prod_{a \in A \text{ monic}} S_{a}(t)^{q^{r}-1}.
\]
This leads to the error estimate in Theorem \ref{Theorem.h-function-as-power-series-in-t}, where we replace the above product by
\[
	(h')^{q}t \prod_{a \text{ monic of degree } 1} S_{a}(t)^{-1}.
\]
So our results are twofold: determining the first few terms of an expansion, and estimating the error, that is the length of the following string
of vanishing coefficients. While for (a) and (c) the first task is easy, as we already dispose of the expansions 
\eqref{Eq.Eisenstein-series-Eqk-1-t-expansion} and \eqref{Eq.Discriminant-form-g-r-product-expansion} and we are reduced to work out the 
error estimates, this is different for (b). Here, the closed form of the starting string given in Theorem \ref{Theorem.gk-as-power-series-in-t} 
is far from obvious, and could be guessed only after calculating a number of examples. Therefore Theorem \ref{Theorem.gk-as-power-series-in-t} 
should be regarded as the principal result of the paper.

We assume throughout that the rank \(r\) is larger or equal to \(3\), which is not a serious restriction. In the rank-\(2\) case there is a 
plentitude of literature with partially much sharper results than those that could be obtained from the methods of the present paper. Again
in the case of rank \(2\), usually an arithmetic normalization of the uniformizer \(t\) is made, where the period \(\overline{\pi}\) of the Carlitz
module (the counterpart to the classical \(2\pi\imath\)) enters. Such a normalization is also possible in the present higher rank case 
(see \cite{Sugiyama18}), but applying it would lead us too far away, and would obscure the overall picture. This is why we decided to use (as in \cite{BassonBreuerPink24} or \cite{Gekeler22-2})
the natural geometric normalization \eqref{Eq.Nowhere-vanishing-holomorphic-function-on-Omega} of \(t\).

\section{Modular forms and their expansions}

(See, e.g. \cite{Gekeler17}, \cite{Gekeler22-2}, \cite{Gekeler88}.)

\subsection{} Throughout, \(\mathds{F} = \mathds{F}_{q}\) is the finite field with \(q\) elements, of characteristic \(p\), and \(r \geq 3\)
a natural number, the \textbf{rank}. Both \(q\) and \(r\) are fixed and usually omitted from notation.

\(A = \mathds{F}[T]\) (resp. \(K = \mathds{F}(T)\)) is the polynomial ring (resp. the field of rational functions) in an indeterminate \(T\).
Further, \(K_{\infty} = \mathds{F}((T^{-1}))\) is the completion of \(K\) at its infinite place, with completed algebraic closure \(C_{\infty}\).
The absolute value \(\lvert \mathbin{.} \rvert\) on \(C_{\infty}\) is normalized such that \(\lvert T \rvert = q\).
\[
	\Omega = \Omega^{r} = \mathds{P}^{r-1}(C_{\infty}) \smallsetminus \bigcup H(C_{\infty})
\]
(where \(H\) runs through the hyperplanes defined over \(K_{\infty}\)) with its structure as a rigid-analytic space denotes the Drinfeld symmetric
space. Homogeneous coordinates \((\omega_{1} : \omega_{2} : \ldots : \omega_{r})\) are normalized such that \(\omega_{r} = 1\), hence
\[
	\Omega = \{ \boldsymbol{\omega} \in C_{\infty}^{r} \mid \omega_{1}, \dots, \omega_{r}~K_{\infty}\text{-linearly independent and } \omega_{r} = 1\}.
\]
\subsection{} The group \(\Gamma \defeq \GL(r,A)\) acts on \(\Omega\) through fractional linear transformations
\begin{equation}
	\gamma \boldsymbol{\omega} = \aut(\gamma, \boldsymbol{\omega})^{-1} \gamma \begin{pmatrix} \omega_{1} \\ \vdots \\ \omega_{r} \end{pmatrix},
\end{equation}
where the right hand side is the matrix product and 
\begin{equation}
	\aut(\gamma, \boldsymbol{\omega}) \defeq \sum_{1 \leq i \leq r} \gamma_{r,i} \omega_{i}.
\end{equation}
A holomorphic function \(f\) on \(\Omega\) is a \textbf{weak modular form} (or \textbf{weakly modular}) of \textbf{weight} \(k \in \mathds{Z}\) and
\textbf{type} \(m \in \mathds{Z}/(q-1)\) if 
\begin{equation}\label{Eq.Weak-modular-form}
	f(\gamma \boldsymbol{\omega}) = \aut(\gamma, \boldsymbol{\omega})^{k} (\det \gamma)^{{-}m} f(\boldsymbol{\omega})
\end{equation}
holds for \(\gamma \in \Gamma\) and \(\boldsymbol{\omega} \in \Omega\).

\subsection{}\label{Subsection.A-lattice}% 
An \(A\)-lattice in \(C_{\infty}\) is a finitely generated (thus free of some rank \(\rho\)) \(A\)-submodule \(\Lambda\) of 
\(C_{\infty}\) which is discrete in the sense that \(\Lambda\) intersects with each ball in at most finitely many points. Hence
\(\Lambda = \bigoplus_{1 \leq i \leq \rho} A\omega_{i}\) with \(\rho\) \(K_{\infty}\)-linearly independent elements \(\omega_{i}\). With each lattice
\(\Lambda\), we associate
\begin{itemize}
	\item its \textbf{exponential function}
	\[
		e^{\Lambda}(z) = z \sideset{}{^{\prime}}\prod_{\lambda \in \Lambda} (1 - z/\lambda)
	\]
	(the primed product \(\sideset{}{^{\prime}}\prod\) (resp. sum \(\sideset{}{^{\prime}}\sum\)) is over the non-zero elements of the index set).
	It is everywhere convergent and has a power series expansion
	\begin{equation}
		e^{\Lambda}(z) = \sum_{i \geq 0} \alpha_{i}(\Lambda) z^{q^{i}};
	\end{equation}
	\item the \textbf{Eisenstein series}
	\[
		E_{k}(\Lambda) = \sideset{}{^{\prime}} \sum_{\lambda \in \Lambda} \lambda^{{-}k}
	\]
	(\(k > 0\) and \(k \equiv 0 \pmod{q-1}\)); otherwise it vanishes);
	\item the \textbf{Drinfeld module} \(\phi^{\Lambda} = C_{\infty}/\Lambda\) of rank \(\rho\).
\end{itemize}
The latter is an exotic \(A\)-module structure given by the \(a\)-th operator polynomial \(\phi_{a}^{\Lambda}\) for each \(0 \neq a \in A\), which
makes the following diagram with exact rows commutative:
\begin{equation}\label{Eq.Exotic-A-module-structure}
	\begin{tikzcd}
		0 \ar[r]	& \Lambda \ar[d, "a"] \ar[r]		& C_{\infty} \ar[d, "a"] \ar[r, "e^{\Lambda}"]	& C_{\infty} \ar[d, "\phi_{a}^{\Lambda}"] \ar[r]	& 0 \\
		0 \ar[r]	& \Lambda \ar[r]					& C_{\infty} \ar[r, "e^{\Lambda}"]				& C_{\infty}	 \ar[r]									& 0
	\end{tikzcd}
\end{equation}
\(\phi^{\Lambda}\) is uniquely determined already by \(\phi_{T}^{\Lambda}\), which has shape 
\begin{equation}\label{Eq.Drinfeld-module}
	\phi_{T}^{\Lambda}(X) = TX + g_{1}(\Lambda)X^{q} + \cdots + g_{\rho}(\Lambda) X^{q^{\rho}}.
\end{equation}
The only restriction on \(\phi_{T}^{\Lambda}\) is that the \textbf{discriminant}
\[
	\Delta(\phi^{\Lambda}) \defeq \Delta(\Lambda) \defeq g_{\rho}(\Lambda)
\]
be non-zero.

\subsection{}\label{Subsection.Ring-of-formal-power-series-in-non-commutative-variable}% 
Let \(C_{\infty}\{\{ \tau \}\}\) be the ring of formal power series in a non-commutative variable \(\tau\) with commutation rule
\begin{align}
	\tau c = c^{q} \tau &&\text{for } c \in C_{\infty},
\end{align}
with sub-ring \(C_{\infty}\{\tau\}\) of polynomials. Then \(C_{\infty}\{\{\tau\}\}\) may be identified via \(\tau^{n} \leadsto z^{q^{n}}\) with the
ring of \(\mathds{F}\)-linear power series of shape \(\sum c_{n}z^{q^{n}}\), where the product of the latter is given by insertion. The map
\begin{equation}
	\phi^{\Lambda} \colon a \longmapsto \phi_{a}^{\Lambda}
\end{equation}
is a homomorphism of \(\mathds{F}\)-algebras from \(A\) to \(C_{\infty}\{\tau\}\), and we can write 
\begin{equation}
	e^{\Lambda}(z) = \sum \alpha_{i}z^{q^{i}} = \sum \alpha_{i}\tau^{i}, \qquad \text{where } \alpha_{i} = \alpha_{i}(\Lambda).
\end{equation}
Define the series \((\beta_{i})_{i \in \mathds{N}_{0}}\) by \(\beta_{i} = {-}E_{q^{i}-1}(\Lambda)\) (with \(E_{0}(\Lambda) = {-}1\), so 
\(\beta_{0} = 1\)) and \(\log^{\Lambda} \in C_{\infty}\{\{\tau\}\}\) by
\begin{equation}
	\log^{\Lambda} = \sum \beta_{i} \tau^{i}.
\end{equation}
Then 
\begin{equation}\label{Eq.Logarithm-series-and-exponential-series-are-inverse}
	e^{\Lambda} \cdot \log^{\Lambda} = \log^{\Lambda} \cdot e^{\Lambda} = 1 \quad \text{in} \quad C_{\infty}\{\{\tau\}\},
\end{equation}
i.e., the two series are mutual inverses. The equation
\[
	\phi_{T}^{\Lambda} \cdot e^{\Lambda} = e^{\Lambda} \cdot T
\]
from \eqref{Eq.Exotic-A-module-structure} leads to 
\begin{equation}\label{Eq.Logarithm-and-Drinfeld-module}
	\log^{\Lambda} \cdot \phi_{T}^{\Lambda} = T \cdot \log^{\Lambda}.
\end{equation}
Comparing coefficients in \eqref{Eq.Logarithm-series-and-exponential-series-are-inverse} gives 
\begin{equation}\label{Eq.Coefficients-for-logarithm}
	\sum_{i+j=k} \alpha_{i}\beta_{j}^{q^{i}} = \sum_{i+j=k} \beta_{i} \alpha_{j}^{q^{i}} = \begin{cases} 1,	&\text{if } k=0, \\ 0,	&\text{if } k > 0. \end{cases}
\end{equation}
Similarly, from \eqref{Eq.Logarithm-and-Drinfeld-module},
\begin{align}\label{Eq.Second-Coefficients-for-logarithm}
	\sum_{i+j=k} \beta_{i} g_{j}^{q^{i}} = T\beta_{k}	&&(k \geq 0),
\end{align}
where \(g_{0} = T\) and all the \(\alpha_{i}\), \(\beta_{j}\), \(g_{k}\) depend on \(\Lambda\). These formulas allow to recursively calculate any
two of the series \((\alpha_{i})\), \((\beta_{j})\), \((g_{k})\) from the third. For example,
\begin{equation}\label{Eq.g_k}
	g_{k} = {-}\sum_{1 \leq j < k} \beta_{k-j}g_{j}^{q^{k-j}} - [k] \beta_{k} = \sum_{1 \leq j < k} E_{q^{k-j}-1}q_{j}^{q^{k-j}} + [k]E_{q^{k}-1},
\end{equation}
where \([k]\) is short for \(T^{q^{k}}-T \in A\) and \(E_{\ell} = E_{\ell}(\Lambda)\).

\subsection{}\label{Subsection.Lattice-Lambda-omega} For \(\boldsymbol{\omega} \in \Omega\) let \(\Lambda_{\boldsymbol{\omega}}\) be the 
lattice \(\sum_{1 \leq i \leq r} A\omega_{i}\). We further let \(\boldsymbol{\omega}' = (\omega_{2}, \dots, \omega_{r}) \in \Omega' = \Omega^{r-1}\) 
and \(\Lambda' = \Lambda_{\boldsymbol{\omega}'} = \sum_{2 \leq i \leq r} A\omega_{i}\). Consider the function on \(\Omega\):
\begin{equation}\label{Eq.Nowhere-vanishing-holomorphic-function-on-Omega}\stepcounter{subsubsection}%
	t(\boldsymbol{\omega}) = (e^{\Lambda'}(\omega_{1}))^{-1}.
\end{equation}
It is well-defined, holomorphic and nowhere vanishing on \(\Omega\), and satisfies 
\begin{equation}\stepcounter{subsubsection}%
	t(\widetilde{\boldsymbol{\omega}}) = t(\boldsymbol{\omega}) \quad \text{if} \quad \widetilde{\boldsymbol{\omega}} = (\omega_{1} + \lambda', \omega_{2}, \dots, \omega_{r}) \quad \text{for} \quad \lambda' \in \Lambda'.
\end{equation}
If \(f\) is a weak modular form then \eqref{Eq.Weak-modular-form} implies that \(f\) satisfies the same rule 
\(f(\widetilde{\boldsymbol{\omega}}) = f(\boldsymbol{\omega})\) and, moreover, that \(f\) has a Laurent expansion
\begin{equation}\stepcounter{subsubsection}%
	f(\boldsymbol{\omega}) = \sum_{n \in \mathds{Z}} a_{n}(\boldsymbol{\omega}') t^{n}(\boldsymbol{\omega})
\end{equation}
for sufficiently small values of \(t\), where the coefficients \(a_{n}\) depend only on \(\boldsymbol{\omega}'\).
\subsubsection{} Now \(f\) is called \textbf{modular} if the coefficients \(a_{n}(\boldsymbol{\omega}')\) vanish identically for \(n < 0\), and is
a \textbf{cusp form} if even \(a_{n}(\boldsymbol{\omega}') \equiv 0\) for \(n \geq 0\).

\begin{Remark}
	\begin{enumerate}[wide, label=(\roman*)]
		\item Suppose that \(f\) has weight \(k\) and type \(m\). Then the coefficient functions \(a_{n}\) are in fact weakly modular of
		weight \(k-n\) and type \(m\), see e.g. \cite{BassonBreuerPink24}.
		\item If \(a_{n} \neq 0\) then \(n \equiv k+m \pmod{q-1}\).
		\item The boundary condition \enquote{\(a_{n} \equiv 0\) if \(n<0\)} for a weak modular form \(f\) to be in fact modular is equivalent
		with \enquote{\(f\) is bounded on the canonical fundamental domain \(\mathbf{F}\) of \(\Gamma\)}, see [14], but we will not use this fact.
	\end{enumerate}	
\end{Remark}

\subsection{} Associating the lattice \(\Lambda_{\boldsymbol{\omega}}\) with \(\boldsymbol{\omega} \in \Omega\), we will regard the 
\(\alpha_{i} = \alpha_{i}(\Lambda_{\boldsymbol{\omega}}) = \alpha_{i}(\boldsymbol{\omega})\), and ditto \(\beta_{i}\), \(g_{i}\), \(E_{k}\), as
functions on \(\Omega\). Then \(\alpha_{i}\), \(\beta_{i}\), \(g_{i}\) are actually modular forms of weight \(q^{i} -1\), and \(E_{k}\) is
modular of weight \(k\), all of type zero.

Define \(M_{k,m}\) as the \(C_{\infty}\)-vector space of modular forms of weight \(k\) and type \(m\), and 
\begin{equation}
	\mathbf{M} \defeq \bigoplus_{\substack{k \in \mathds{N}_{0} \\ m \in \mathds{Z}/(q-1)}} M_{k,m}
\end{equation}
as the bigraded \(C_{\infty}\)-algebra of modular forms, with sub-algebra
\begin{equation}
	\mathbf{M}_{0} \defeq \bigoplus_{k \in \mathds{N}_{0}} M_{k,0}
\end{equation}
of forms of type \(0\). Then:

\begin{Theorem}[{\cite{Gekeler17}, \cite{BassonBreuerPink24}}] ~
	\begin{enumerate}[label=\(\mathrm{(\roman*)}\)]
		\item \(\mathbf{M}_{0}\) is the polynomial ring \(C_{\infty}[g_{1}, \dots, g_{r}]\) in the algebraically independent forms \(g_{1}, \dots, g_{r}\);
		\item \(\mathbf{M} = \mathbf{M}_{0}[h]\), where \(h \in M_{(q^{r}-1)/(q-1),1}\) is a \((q-1)\)-th root of the discriminant 
		\(\Delta = g_{r}\).
	\end{enumerate}
\end{Theorem}

\begin{Remarks}
	\begin{enumerate}[wide, label=(\roman*)]
		\item The various \(M_{k,m}\) are linearly independent as functions on \(\Omega\); hence their sum is in fact direct.
		\item We choose the normalization of \(h\) such that \(h^{q-1} = ({-}1)^{r-1}\Delta\). This allows the formula
		\eqref{Eq.Discriminant-form-g-r-product-expansion} for \(h\).
		\item As results from \eqref{Eq.Second-Coefficients-for-logarithm} and \eqref{Eq.g_k}, we can also write
		\[
			\mathbf{M}_{0} = C_{\infty}[\alpha_{1}, \dots, \alpha_{r}] = C_{\infty}[\beta_{1}, \dots, \beta_{r}] = C_{\infty}[E_{q-1}, \dots, E_{q^{r}-1}].
		\]
	\end{enumerate}	
\end{Remarks}

\subsection{}\label{Subsection.Drinfeld-A-module-of-rank-r}% 
Given \(\boldsymbol{\omega} \in \Omega\), we let \(\phi^{\boldsymbol{\omega}}\) be the Drinfeld \(A\)-module of rank \(r\) associated
with the lattice \(\Lambda_{\boldsymbol{\omega}}\), i.e., \(\phi^{\boldsymbol{\omega}} = \phi^{\Lambda_{\boldsymbol{\omega}}}\). Similarly, writing
\(\boldsymbol{\omega} = (\omega_{1}, \boldsymbol{\omega}')\) as in \ref{Subsection.Lattice-Lambda-omega}, 
\(\phi^{\boldsymbol{\omega}'} = \phi^{\Lambda_{\boldsymbol{\omega}'}}\), of rank \(r'=r-1\). Quite generally, we write objects related to
\(\phi^{\boldsymbol{\omega'}}\), for example functions of type \(\alpha_{i}\), \(\beta_{i}\), \(g_{i}\), \dots with a prime
\(\alpha_{i}'\), \(\beta_{i}'\), \(g_{i}'\), \dots to distinguish them from the \enquote{same} objects related to \(\phi^{\boldsymbol{\omega}}\).

For \(0 \neq a \in A\), of degree \(d\) and leading coefficient \(\sgn(a) \in \mathds{F}^{*}\), the operator polynomial 
\(\phi^{\boldsymbol{\omega}'}_{a}(X)\) (see \eqref{Eq.Drinfeld-module}) of \(\phi^{\boldsymbol{\omega}'}\) has shape
\begin{equation}
	\phi_{a}^{\boldsymbol{\omega}'}(X) = aX + \prescript{}{a}\ell_{1}'(\boldsymbol{\omega}')X^{q} + \prescript{}{a}\ell_{2}'(\boldsymbol{\omega}')X^{q^{2}} + \cdots + \prescript{}{a}\ell_{(r-1)d}'(\boldsymbol{\omega}')X^{q^{(r-1)d}},
\end{equation}
with last coefficient
\begin{equation}\label{Eq.Operator-polynomial-last-coefficient}
	\prescript{}{a}\ell_{(r-1)d}'(\boldsymbol{\omega}') = \sgn(a) \Delta'(\boldsymbol{\omega}')^{e} \eqdef \Delta_{a}'(\boldsymbol{\omega}') \neq 0,
\end{equation}
where the exponent \(e\) is \((q^{(r-1)d}-1)/(q^{r-1}-1)\). Define
\begin{align}\label{Eq.Definition-of-SaX}
	S_{a}(X) 	&\defeq \Delta_{a}'(\boldsymbol{\omega}')^{-1}X^{q^{(r-1)d}} \phi_{a}^{\boldsymbol{\omega}'}(X^{-1}) \\
				&= 1 + \frac{\prescript{}{a}\ell_{(r-1)d-1}'}{\Delta_{a}'} X^{q^{(r-1)d} - q^{(r-1)d - 1}} + \cdots + \frac{a}{\Delta_{a}'} X^{q^{(r-1)d} - 1}. \nonumber
\end{align}
(We omit reference to \(\boldsymbol{\omega}'\) of \(S_{a}(X)\) and of its coefficients \(\frac{\prescript{}{a}\ell_{i}'}{\Delta_{a}'}\).) Let
\(\mathcal{O}' = \mathcal{O}(\Omega')\) be the ring of holomorphic functions on \(\Omega' = \Omega^{r-1}\). As \(\phi' = \phi^{\boldsymbol{\omega}'}\)
varies with \(\boldsymbol{\omega}'\) over \(\Omega'\), \(S_{a}(X) \in \mathcal{O}'[X]\) is a polynomial with coefficients in \(\mathcal{O}'\). If
\(a \in \mathds{F}^{*}\) then \(S_{a}(X) = 1\); if \(a = T+c\) with \(c \in \mathds{F}\) then
\begin{equation}
	S_{a}(X) = 1 + \frac{g_{r-1}'}{\Delta'} X^{q^{r-1}-q^{r-2}} + \cdots + \frac{g_{1}'}{\Delta'}X^{q^{r-1}-q} + \frac{T+c}{\Delta'} X^{q^{r-1}-1}.
\end{equation}
With notation as before, let \(t_{a}\) be the function
\begin{equation}\label{Eq.Definition-of-ta-omega}
	t_{a}(\boldsymbol{\omega}) = t(a\omega_{1}, \boldsymbol{\omega}') = (\Delta_{a}')^{-1} t^{q^{(r-1)d}}/S_{a}(t),
\end{equation}
where the right hand side is regarded as a power series in \(t\) with coefficients in \(\mathcal{O}'\).

By means of \(S_{a}\) and \(t_{a}\), we may describe the \(t\)-expansions of some modular forms, as follows.

\begin{Theorem}[{\cite{Gekeler88}, \cite{Gekeler25}}]\label{Theorem.Eisenstein-series-Eqk-1-t-expansion}
	The special Eisenstein series \(E_{q^{k}-1}\) has the \(t\)-expansion
	\begin{equation}\label{Eq.Eisenstein-series-Eqk-1-t-expansion}
		E_{q^{k}-1}(\boldsymbol{\omega}) = E_{q^{k}-1}'(\boldsymbol{\omega}') - \sum_{a \in A \text{ monic}} G_{q^{k}-1}(t_{a}),
	\end{equation}	
	where \(G_{q^{k}-1}(X)\) is the polynomial
	\begin{equation}\label{Eq.Polynomial-Gqk-1-X}
		G_{q^{k}-1}(X) = \sum_{0 \leq i < k} \beta_{i}' X^{q^{k}-q^{i}}
	\end{equation}
	with \(\beta_{i}' = \beta_{i}'(\boldsymbol{\omega}') = {-}E_{q^{i}-1}'(\boldsymbol{\omega}')\). It converges for \(\lvert t \rvert\) small enough,
	locally uniformly in \(\boldsymbol{\omega}'\).
\end{Theorem}

\begin{Theorem}[{\cite{Gekeler85}, \cite{Basson17}, \cite{Gekeler25}}]
	The discriminant form \(\Delta = g_{r}\) has a product expansion
	\begin{equation}
		\Delta(\boldsymbol{\omega}) = {-}\Delta'(\boldsymbol{\omega}')^{q}t^{q-1} \prod_{a \in A \text{ monic}} S_{a}(t)^{(q^{r}-1)(q-1)}.
	\end{equation}	
	Accordingly, the function \(h\) (with the condition \(h^{q-1} = ({-}1)^{r-1}\Delta\)) may be chosen such that
	\begin{equation}\label{Eq.Discriminant-form-g-r-product-expansion}
		h(\boldsymbol{\omega}) = h'(\boldsymbol{\omega}')^{q}t \prod_{a \text{ monic}} S_{a}(t)^{q^{r}-1}.
	\end{equation} 
\end{Theorem}

We will use these to determine the first few coefficients \(a_{n}\) of 
\[
	f(\boldsymbol{\omega}) = \sum_{n \geq 0} a_{n}(\boldsymbol{\omega}') t^{n}(\boldsymbol{\omega})
\]
for \(f \in \{g_{1}, \dots, g_{r}, h, E_{q-1}, E_{q^{2}-1}, \dots \}\).

\section{The special Eisenstein series \(E_{q^{k}-1}\)}

\subsection{} The \textbf{support} \(\supp(P)\) of a polynomial or power series \(P(X)\) is the set of indices \(n \in \mathds{N}_{0}\) such that
the \(n\)-th coefficient of \(P\) is non-zero. The vanishing order or briefly \textbf{order} of \(P\) is \(o(P) \defeq \min \supp(P)\). We write
\(P(X) = o(X^{n})\) if \(o(P) \geq n\). For example,
\begin{equation}
	o(G_{q^{k}-1}) = q^{k}-q^{k-1}
\end{equation}
(as long as \(\beta'_{k-1} \neq 0\), see \eqref{Eq.Polynomial-Gqk-1-X}). 

An \textbf{affine-linear} map of \(\mathds{F}\)-vector spaces is the sum of a linear and a constant map. The crucial property for our purposes is
the easy observation: Let \(f \colon V \to W\) be an affine-linear map of \(\mathds{F}\)-vector spaces with \(0 < \dim V < \infty\). Then
\begin{equation}\label{Eq.Affine-linear-map-on-finite-dimensional-IF-space}
	\sum_{v \in V} f(v) = 0
\end{equation}
except possibly in the case where \(q = 2\) and \(\dim V = 1\).

\subsection{} As in \ref{Subsection.Drinfeld-A-module-of-rank-r}, we let \(a \in A\) be a monic (i.e., \(\sgn(a) = 1\)) element of degree \(d \geq 1\).
Then
\begin{align}
	o(S_{a}(X) -1)	&= q^{(r-1)d} - q^{(r-1)d -1} \label{Eq.SaX-for-monic-a}
	\intertext{and}
	o(t_{a})		&= q^{(r-1)d},	
\end{align}
\(t_{a}\) being considered as a power series in \(t\). (We have used the fact that neither of the coefficient functions 
\(\prescript{}{a}\ell_{i}'(\boldsymbol{\omega}')\) vanishes identically, see e.g. \cite{Gekeler22}.)

Write
\begin{equation}\label{Eq.a-linear-combination}
	a = \sum_{0 \leq i \leq d} c_{i}T^{i} \quad \text{with} \quad c_{i} \in \mathds{F}, c_{d}= 1.
\end{equation}
As
\[
	\phi_{a}^{\boldsymbol{\omega}'} = \sum c_{i} \phi_{T^{i}}^{\boldsymbol{\omega}'} \quad \text{and}\quad \phi_{T^{i}}^{\boldsymbol{\omega}'} = \phi_{T}^{\boldsymbol{\omega}'} \circ \phi_{T}^{\boldsymbol{\omega}'} \circ \cdots \circ \phi_{T}^{\boldsymbol{\omega}'}
\]
(\(i\) factors, successively inserted, \(\phi_{1}^{\boldsymbol{\omega}'}(X) = X\)), the coefficients of \(\phi_{a}^{\boldsymbol{\omega}'}(X)\) are
affine-linear functions of \((c_{0}, \dots, c_{d-1})\), and the same holds for the coefficients of the reciprocal polynomial \(S_{a}(X)\)
(see \eqref{Eq.Definition-of-SaX}). Here we use that by \eqref{Eq.Operator-polynomial-last-coefficient},
\begin{equation}\label{Eq.Coefficients-for-Drinfeld-modular-form-phi-a-omega-prime}
	\Delta_{(d)}' \defeq \Delta_{a}' = {\Delta'}^{(q^{(r-1)d}-1)/(q^{r-1}-1))}
\end{equation}
depends only on the degree \(d\) of \(a\).

\subsection{} To get control of the Eisenstein series \(E_{q^{k}-1}\) via Theorem \ref{Theorem.Eisenstein-series-Eqk-1-t-expansion}, we must evaluate
the expression
\[
	\sum_{a \in A \text{ monic}} t_{a}^{q^{k}-q^{i}} = \Big( \sum_{a} t_{a}^{q^{k-i}-1} \Big)^{q^{i}}.
\]
Therefore we define 
\begin{align}\label{Eq.Definition-Theta-j-d}
	\Theta(j,d) \defeq \sum_{\substack{a \in A \text{ monic} \\ \deg a = d}} t_{a}^{q^{j}-1} &&(j,d \geq 1),	
\end{align}
a power series in \(t\) with coefficients in \(\mathcal{O}'\). Our aim is to get good lower estimates for the order \(o(\Theta(j,d))\). By 
the definition \eqref{Eq.Definition-of-ta-omega} of \(t_{a}\) and \eqref{Eq.Coefficients-for-Drinfeld-modular-form-phi-a-omega-prime},
\begin{equation}\label{Eq.Estimate-for-Theta-j-d}
	o(\Theta(j,d)) = q^{(r-1)d}(q^{j}-1) + o(\Sigma(j,d)),
\end{equation}
where
\begin{equation}
	\Sigma(j,d) \defeq \sum_{\substack{a \text{ monic} \\ \deg a = d}} S_{a}^{1-q^{j}} \in \mathcal{O}'[[t]].
\end{equation}
\subsection{} We will evaluate \(\Sigma(j,d)\) via the scheme
\begin{equation}
	\Sigma(j,d) = \sum_{a} S_{a}^{1-q^{j}} = \sum_{a} \frac{S_{a}}{S_{a}^{q^{j}}} = \sum_{a} S_{a}(1 - (S_{a}^{q^{j}}-1) + (S_{a}^{q^{j}}-1)^{2} - \cdots ).
\end{equation}
As by \eqref{Eq.SaX-for-monic-a}
\[
	S_{a}^{q^{j}}(t) = 1 + o(t^{q^{(r-1)d+j} - q^{(r-1)d+j-1}}),
\]
the supports of \(S_{a}(t)\), \(S_{a}^{q^{j}}(t)-1\), and \((S_{a}^{q^{j}}-1)^{2}\) are mutually disjoint. Therefore
\begin{equation}\label{Eq.Support-of-Sa-1-qj}
	\supp(S_{a}^{1-q^{j}}) = \supp(S_{a}) \cupdot \supp(S_{a}^{q^{j}}-1) \cupdot R
\end{equation}
with some subset \(R \subset \mathds{N}\) and \(\min(R) = 2(q^{(r-1)d+j}-q^{(r-1)d+j-1})\). Again, this description depends only on the degree of
\(a\), but not on \(a\) itself. For \(n \in \mathds{N}_{0}\) let \(\kappa_{n,j}(a)\) be the \(n\)-coefficient of \(S_{a}^{1-q^{j}}\). It follows
from \eqref{Eq.a-linear-combination} and \eqref{Eq.Support-of-Sa-1-qj} that the map
\begin{equation}\label{Eq.Affine-linear-map-kappa-n-j}
	(c_{0}, \dots, c_{d-1}) \longmapsto \kappa_{n,j}\Big( \sum_{0 \leq i < d} c_{i}T^{i} + T^{d} \Big)
\end{equation}
is in fact affine-linear if \(n \in \supp(S_{a}) \cup \supp(S_{a}^{q^{j}} - 1)\). Looking at 
\eqref{Eq.Affine-linear-map-on-finite-dimensional-IF-space}, we get our first result.

\begin{Proposition}\label{Proposition.Vanishing-of-kappa-n-j}
	Suppose that \(q>2\) or \(d>1\). Then \(\sum_{a \text{ monic}, \deg a = d} \kappa_{n,j}(a)\) vanishes for \(n < 2(q^{(r-1)d+j}-q^{(r-1)d+j-1})\).
	Hence \(o(\Sigma(j,d)) \geq 2(q^{(r-1)d+j}-q^{(r-1)d+j-1})\). If \(q=2\) and \(d=1\),
	\[
		\Sigma(j,1) = S_{T}^{1-q^{j}} + S_{T+1}^{1-q^{j}} = {\Delta'}^{{-}1} t^{2^{r-1}-1} + {\Delta'}^{{-}2^{j}} t^{2^{r+j-1} - 2^{j}} + o(t^{2^{r+j-1}}),
	\]
	where \(\Delta' \in \mathcal{O}'\) is the discriminant form of rank \(r-1\).
\end{Proposition}

\begin{proof}
	The first part follows from \eqref{Eq.Affine-linear-map-kappa-n-j}, \eqref{Eq.Support-of-Sa-1-qj} and 
	\eqref{Eq.Affine-linear-map-on-finite-dimensional-IF-space}, the second from inspection.
\end{proof}

\begin{Remark}
	It is easy to check that the various \(\supp(( S_{a}^{q^{j}}-1)^{\ell})\) are mutually disjoint as long as \(\ell \leq q\), and so
	\[
		\supp( S_{a}^{1-q^{j}}) = \supp(S_{a}) \bigcupdot_{1 \leq \ell \leq q} \supp( (S_{a}^{q^{j}-1})^{\ell}) \cupdot R'
	\]
	with some remainder \(R' \subset \mathds{N}\). For \(n \in \supp( (S_{a}^{q^{j}}-1)^{\ell})\) the function \((c_{0}, \dots, c_{d-1}) \mapsto \kappa_{n,j}( \sum c_{i}T^{i} + T^{d})\) of \eqref{Eq.Affine-linear-map-kappa-n-j} is \(m\)-polynomial for some \(m\) with
	\(0 \leq m \leq \ell\). Here, \(f\) \(m\)-polynomial means \(f(v) = \sum_{0 \leq i \leq m} f_{i}(v)\), where \(f_{i}(cv) = c^{i}f(v)\)
	for \(c \in \mathds{F}\).
	
	The vanishing of \(\sum_{v \in V} f(v)\) can be shown under assumptions on the size of \(q\). E.g., \(\sum_{v \in V} f_{2}(v) = 0\) for a
	\(2\)-homogeneous map is assured if there exists \(c \in \mathds{F}^{*}\) such that \(c^{2} \neq 1\), i.e., \(q>3\). Hence, making assumptions
	on the size of \(q\), the bound in Proposition \ref{Proposition.Vanishing-of-kappa-n-j} may be sharpened.
\end{Remark}

\begin{Corollary}\label{Corollary.Better-estimate-for-Theta-j-d}
	The quantity \(\Theta(j,d)\) of \eqref{Eq.Definition-Theta-j-d} satisifes \(\Theta(j,d) = o(t^{N})\) with \(N = 3q^{(r-1)d+j} - 2q^{(r-1)d+j-1} - q^{(r-1)d}\) if \((q,d) \neq (2,1)\) and \(N = 2^{r+j-1}-1\) if \((q,d) = (2,1)\).	
\end{Corollary}

\begin{proof}
	\ref{Proposition.Vanishing-of-kappa-n-j} + \eqref{Eq.Estimate-for-Theta-j-d}.	
\end{proof}

Now we can describe the \(t\)-expansion of the special Eisenstein series as follows.

\begin{Theorem}\label{Theorem.Expansion-for-Eqk-1}
	We have
	\begin{equation}
		E_{q^{k}-1}(\boldsymbol{\omega}) = E_{q^{k}-1}'(\boldsymbol{\omega}') - G_{q^{k}-1}(t) + o(t^{N}) = \sum_{0 \leq i \leq k} E_{q^{i}-1}'(\boldsymbol{\omega}')t^{q^{k}-q^{i}} + o(t^{N})
	\end{equation}	
	with \(N = 3(q-1)q^{k+r-2}\) for \(q>2\) and \(N = 2^{k-1}(2^{r}-1)\) for \(q=2\).
\end{Theorem}

\begin{proof}
	Omitting in \eqref{Eq.Eisenstein-series-Eqk-1-t-expansion} the terms \(G_{q^{k}-1}(t_{a})\) with \(\deg a > 0\) leads to an error
	\[
		o\Big( \sum_{\substack{a \text{ monic}\\ \deg a = 1}} t_{a}^{q^{k}-q^{k-1}} \Big) = o(t^{q^{k-1} \cdot o(\Theta(1,1))}).
	\]	
	By Corollary \ref{Corollary.Better-estimate-for-Theta-j-d}, the exponent \(N\) results as stated.
\end{proof}

\subsection{} What can be said about the expansions of non-special Eisenstein series \(E_{k}\), where \(0 < k \equiv 0 \pmod{q-1}\)? Again,
there is an expansion like \eqref{Eq.Eisenstein-series-Eqk-1-t-expansion}
\begin{equation}
	E_{k}(\boldsymbol{\omega}) = E_{k}'(\boldsymbol{\omega}') - G_{k}(t) - \sum_{\substack{a \text{ monic} \\ \deg a \geq 1}} G_{k}(t_{a})
\end{equation}
with a monic polynomial \(G_{k}(X) \in \mathcal{O}'[X]\) of degree \(k\), see \cite{Gekeler22-3}. But for \(k\) not of the special form \(k = q^{j}-1\), \(G_{k}\)
behaves erratic, and its order \(\nu(k) \defeq o(G_{k}(X))\) doesn't grow with \(k\), as is the case with \(G_{q^{k}-1}\). At least,
\[
	\nu = \nu(k) \geq \max(2,q-1),
\]
so we have the crude estimate
\[
	o\Big( \sum_{\substack{a \text{ monic}\\ \deg a = d}} t_{a}^{\nu} \Big) = \nu q^{(r-1)d} + o \Big( \sum_{a} S_{a}^{{-}\nu} \Big)
\]
with 
\begin{align}
	o\Big( \sum_{a} S_{a}^{{-}\nu} \Big)		&\geq o(S_{a}^{{-}\nu} ) &&(\text{one fixed monic \(a\) of degree \(d\)}) \nonumber \\
											&\geq o(S_{a}) \nonumber \\
											&=q^{(r-1)d} - q^{(r-1)d-1}. \nonumber
	\intertext{Thus}
	o\Big( \sum_{\substack{a \text{ monic} \\ \deg a = d}} t_{a}^{\nu} \Big)		&= (\nu+1) q^{(r-1)d} - q^{(r-1)d-1} \nonumber \\
																				&= q^{(r-1)d+1} - q^{(r-1)d-1} \nonumber 
	\intertext{and finally}
	E_{k}(\boldsymbol{\omega})				&= E_{k}'(\boldsymbol{\omega}') - G_{k}(t) + o(t^{N})
\end{align}
with \(N = q^{r}-q^{r-2}\). This formula is meaningful for such \(k\) less than \(N\) for which \(G_{k}(X)\) is accessible.

\section{The coefficient forms \(g_{k}\)}

Our aim is to show the following formula for the starting terms of the \(t\)-expansions of the forms \(g_{k}\) (\(1 \leq k \leq r\)). As before,
\(g_{i}' = g_{i}'(\boldsymbol{\omega}')\) is the coerresponding coefficient form in rank \(r' = r-1\).

\begin{Theorem}\label{Theorem.gk-as-power-series-in-t}
	As a power series in \(t\), we have
	\begin{align*}
		g_{k} 	&= g_{k}' - {g_{k-1}'}^{q}t^{q-1} + g_{k-1}'t^{q^{k}-q^{k-1}} - {g_{k-2}'}^{q} t^{q^{k}-q^{k-1}+q-1} + g_{k-2}'t^{q^{k}-q^{k-2}} - \\
				&\qquad  \cdots - {g_{1}'}^{q} t^{q^{k}-q^{2}+q-1} + g_{1}'t^{q^{k}-q} - [1]t^{q^{k}-1} + o(t^{N})
	\end{align*}
	with \(N = 3(q-1)q^{r-1}\) if \(q>2\), \(N = 3 \cdot 2^{r-1}-1\) if \(q=2\), and \([1] = T^{q}-T\). In closed form, 
	\begin{equation}\label{Eq.Closed-form-gk}
		g_{k} = g_{k}' + \sum_{1 \leq i \leq k-1} \big( g_{i}'t^{q^{k}-q^{i}} - {g_{i}'}^{q} t^{q^{k}-q^{i+1}+q-1} \big) - [1]t^{q^{k}-1} + o(t^{N}).
	\end{equation}
\end{Theorem}

Note that \eqref{Eq.Closed-form-gk} gives the \(2k-2\) intermediate terms in reverse \(t\)-order.

\begin{Corollary}
	The coefficient of \(t^{n}\) in \(g_{k}\) (a priori a weak modular form of weight \(q^{k}-1-n\) and type \(0\) in \(\boldsymbol{\omega}'\))
	is in fact a modular form for \(n < N\) (and vanishes for \(q^{k}-1<n<N\)).	
\end{Corollary}

A possible approach to Theorem \ref{Theorem.gk-as-power-series-in-t} would be to use \eqref{Eq.Second-Coefficients-for-logarithm}, which recursively
determines \(g_{k}\) from the \(g_{i}\) with \(i<k\) and the \(E_{q^{j}-1}\), and our knowledge of the expansion of \(E_{q^{i}-1}\). However, the
strategy developed below is both less complicated and yields stronger results. 

\subsection{} Recall that \(\mathcal{O}'\) is the ring of holomorphic functions on 
\[
	\Omega' = \Omega^{r-1} = \{ \boldsymbol{\omega}' = (\omega_{2}, \ldots, \omega_{r}) \in C_{\infty}^{r-1} \mid \omega_{i} ~ K_{\infty}\text{-linearly independent}, \omega_{r} = 1\}.
\]
As usual, we write \(\boldsymbol{\omega} = (\omega_{1}, \boldsymbol{\omega}')\) for \(\boldsymbol{\omega} \in \Omega\). We regard the generic Drinfeld
module \(\phi = \phi^{\boldsymbol{\omega}}\) (\(\boldsymbol{\omega}\) varying through \(\Omega\)) as a Tate object (see, e.g. \cite{Papikian23}, and \cite{Gekeler88} for
the case of rank two) over the ring
\begin{equation}\label{Eq.Definition-Ring-R}
	\mathcal{R} \defeq \mathcal{O}'((t))
\end{equation}
of formal Laurent series in \(t\) over \(\mathcal{O}'\), endowed with the \(t\)-adic topology. That is, the coefficients of \(\phi\) are replaced 
with their respective \(t\)-expansions. Let \(\phi' = \phi^{\boldsymbol{\omega}'}\) be the generic Drinfeld module of rank \(r-1\) base extended
to \(\mathcal{R}\), and let
\begin{equation}
	L \defeq \langle t^{{-}1} \rangle
\end{equation}
be the \(A\)-lattice generated by \(t^{{-}1}\) via \(\phi'\), i.e.,
\[
	L = \{ \phi_{a}'(t^{{-}1}) \mid a \in A \}.
\]
Let \(a \in A\) have degree \(d\). For fixed \(\boldsymbol{\omega}\) (and thus also for \(\boldsymbol{\omega}\) varying),
\begin{align*}
	\phi_{a}'(t^{{-}1})	&= \phi_{a}'(e^{\Lambda'_{\boldsymbol{\omega}'}}(\omega_{1})) \\
						&= e^{\Lambda'_{\boldsymbol{\omega}'}}(a \omega_{1}) = t_{a}^{{-}1}(\boldsymbol{\omega}) = \Delta_{(d)}'(\boldsymbol{\omega}') t(\boldsymbol{\omega})^{{-}q^{(r-1)d}}(\boldsymbol{\omega}) S_{a}(t(\boldsymbol{\omega})).	
\end{align*}
Hence \(\phi_{a}'(t^{{-}1})\) is \enquote{large} in \(\mathcal{R}\) and grows with \(d\), and \(L\) is a lattice in the sense of 
\ref{Subsection.A-lattice}. As in \ref{Subsection.Ring-of-formal-power-series-in-non-commutative-variable}, we define the function
\begin{equation}
	\begin{split}
		e^{L} \colon \mathcal{R}	&\longrightarrow \mathcal{R}, \\
							z		&\longmapsto z \sideset{}{^{\prime}} \prod_{\lambda \in L} (1 - z/\lambda)
	\end{split}
\end{equation}
where the product converges everywhere and defines an \(\mathds{F}\)-linear function. We have properties similar to those in \ref{Subsection.A-lattice}
and \ref{Subsection.Ring-of-formal-power-series-in-non-commutative-variable}; in particular,
\begin{itemize}
	\item \(e^{L}(z) = \sum_{i \geq 0} \alpha_{i}(L)z^{q^{i}} = \sum_{i \geq 0} \alpha_{i}(L)\tau^{i} \in \mathcal{R}\{\{\tau\}\}\);
	\item there is a formal inverse \(\log^{L} = \sum_{i \geq 0} \beta_{i}(L)\tau^{i}\) of \(e^{L}\) in \(\mathcal{R}\{\{\tau\}\}\);
	\item \(\beta_{i}(L) = {-}E_{q^{i}-1}(L)\) with \(E_{q^{i}-1}(L) = \sideset{}{^{\prime}}\sum_{\lambda \in L} {\lambda'}^{1-q^{i}}\) if
	\(i > 0\) and \(E_{0}(L) = {-}1\),
\end{itemize}
and thus the relations \eqref{Eq.Coefficients-for-logarithm} for the present data. Instead of diagram \eqref{Eq.Exotic-A-module-structure}, we dispose
for each \(0 \neq a \in A\) of the commutative diagram with exact rows 
\begin{equation}
	\begin{tikzcd}
		0 \ar[r] 	& L \ar[d, "\phi_{a}'"] \ar[r]	& \mathcal{R} \ar[d, "\phi_{a}'"] \ar[r, "e^{L}"]		& \mathcal{R} \ar[d, "\phi_{a}"] \\
		0 \ar[r]	& L	\ar[r]						& \mathcal{R} \ar[r, "e^{L}"]							& \mathcal{R}.
	\end{tikzcd}
\end{equation}
That is, \(e^{L}\) is an \enquote{analytic homomorphism} of the Drinfeld module \(\phi'\) to \(\phi\). 

From \(\phi_{T} \circ e^{L} = e^{L} \circ \phi_{T}'\), we find
\begin{equation}
	\log^{L} \circ \phi_{T} = \phi_{T}' \circ \log^{L},
\end{equation}
a substitute for \eqref{Eq.Logarithm-and-Drinfeld-module}. Comparing coefficients and isolating the term \(g_{k}\) yields the recursion formula
(analogous with \eqref{Eq.g_k}):
\begin{align}\label{Eq.Recursion-formula-gk}
	g_{k} = \sum_{1 \leq i \leq k} E_{q^{i}-1}(L) q_{k-i}^{q^{i}} - \sum_{0 \leq i \leq k} g_{i}' E_{q^{k-i}-1}^{q^{i}}(L) &&(k \geq 1)
\end{align}
for the coefficients \(g_{k}\) of \(\phi_{T}\) (and, of course, the \(g_{i}'\) are those of \(\phi_{T}'\)). Our proof of Theorem 
\ref{Theorem.gk-as-power-series-in-t} will be based on \eqref{Eq.Recursion-formula-gk} and the next result.

\begin{Proposition}\label{Proposition.E-qi-1-in-R}
	In the ring \(\mathcal{R}\) we have 
	\begin{align}
		E_{q^{i}-1}(L) 	&= {-}t^{q^{i}-1} + o(t^{N}) \label{Eq.E-qi-1-in-R}
		\intertext{with}
		N				&= \begin{cases}3q^{r+i-1} - (2q^{i-1}+1)q^{r-1},	&\text{if } q > 2, \\ 2^{r+i} - 2^{r-1} - 1,		&\text{if } q=2.\end{cases}
	\end{align}	
\end{Proposition}

\begin{proof}
	\begin{align*}
		E_{q^{i}-1}(L)	&= \sideset{}{^{\prime}} \sum_{\lambda \in \Lambda} \lambda^{1-q^{i}} \\
						&= \sideset{}{^{\prime}} \sum_{a} \phi_{a}'(t^{{-}1})^{1-q^{i}} \\
						&= {-}\sum_{d \in \mathds{N}_{0}} \sum_{\substack{a \text{ monic} \\ \deg a = d}} \phi_{a}'(t^{{-}1})^{1-q^{i}} \\
						&= {-}\sum_{d \in \mathds{N}_{0}} t^{{-}q^{(r-1)d}(1-q^{i})}	 {\Delta_{(d)}'}^{1-q^{i}} \sum_{\substack{a \text{ monic}\\ \deg a = d}} S_{a}(t)^{1-q^{i}} \\
						&= {-}t^{q^{i}-1} - \sum_{d \in \mathds{N}} {\Delta_{(d)}'}^{1-q^{i}} t^{q^{(r-1)d}(q^{i}-1)} \Sigma(i,d) \\
						&= {-}t^{q^{i}-1} + o(t^{N})
	\end{align*}
	by Corollary \ref{Corollary.Better-estimate-for-Theta-j-d}.
\end{proof}

\begin{proof}[Proof of Theorem \ref{Theorem.gk-as-power-series-in-t}]
	\eqref{Eq.Recursion-formula-gk} and \eqref{Eq.E-qi-1-in-R} yield
	\begin{align*}
		g_{1}	&= E_{q-1}(L) T^{q} - TE_{q-1}(L) - g_{1}'E_{0}^{q}(L) \\
				&= [1]E_{q-1}(L) + g_{1}' = g_{1}' - [1]t^{q-1} + o(t^{N})	
	\end{align*}
	in accordance with the assertion. (By the way, this agrees up to a better error term for \(q=2\) with the result Theorem 
	\ref{Theorem.Expansion-for-Eqk-1} for \(E_{q-1}\), taking \(g_{1} = [1]E_{q-1}\) into account.) Note that all the error terms
	\(o(t^{N'})\) that occur in Proposition \ref{Proposition.E-qi-1-in-R} are such that \(N' \geq N = 3(q-1)q^{r-1}\) resp.
	\(3 \cdot 2^{r-1} - 1\). Therefore, we let \enquote{\(\equiv\)} denote congruence modulo \(t^{N}\) in \(\mathcal{R}\). Assume that
	\eqref{Eq.Closed-form-gk} holds for \(k'\) with \(1 \leq k' < k\). By Proposition \ref{Proposition.E-qi-1-in-R} and 
	\eqref{Eq.Recursion-formula-gk},
	\[
		g_{k} \equiv {-}\sum_{1 \leq i \leq k} t^{q^{i}-1} g_{k-i}^{q^{i}} + \sum_{0 \leq i \leq k} g_{i}'t^{q^{k}-q^{i}}.
	\]
	From the induction hypothesis, the latter is congruent to 
	\begin{align*}
		& \sum_{0 \leq i \leq k} g_{i}'t^{q^{k}-q^{i}} - t^{q^{k}-1}T^{q^{k}} \\
		&\qquad - \sum_{1 \leq i < k} t^{q^{i}-1} \Big\{ g_{k-i}' + \sum_{1 \leq j \leq k-i} \big( g_{j}'t^{q^{k-i}-q^{j}} - g_{j}'t^{q^{k-i}-q^{j+1}+q-1} \big) - [1]t^{q^{k-i}-1} \Big\}^{q^{i}}.
	\end{align*}
	As the result of a lengthy but straightforward calculation (which we omit; several telescoping sums occur, thereby reducing the double sums to
	simple sums), this expression turns out to equal the formula \eqref{Eq.Definition-Ring-R} for \(g_{k}\).
\end{proof}

\section{The \(h\)-function}

\subsection{} We derive similar results for the \(h\)-function
\begin{equation}\label{Eq.Product-h}
	h = {h'}^{q} t \prod_{a \in A \text{ monic}} S_{a}(t)^{q^{r}-1}.
\end{equation}
For \(d \in \mathds{N}\) put
\begin{equation}
	\Pi(d)(X) \defeq \prod_{\substack{a \text{ monic} \\ \deg a = d}} S_{a}(X).
\end{equation}
For typographical reasons, we work with \(s \defeq r-1 =\) rank of \(\phi'\) instead of \(r\). For \(a = T+c\) (\(c \in \mathds{F}\)) of degree 1,
\begin{equation}
	\phi_{a}'(X) = aX + g_{1}'X^{q} + \cdots + g_{s-1}'X^{q^{s-1}} + \Delta' X^{q^{s}},
\end{equation}
where only the coefficient \(a\) of \(X\) depends on \(a\). Thus
\begin{equation}
	S_{a}(X) = 1 + \frac{g_{s-1}}{\Delta'} X^{q^{s}-q^{s-1}} + \cdots + \frac{g_{1}'}{\Delta} X^{q^{s}-q} + \frac{a}{\Delta'} X^{q^{s}-1} = S_{T}(X) + cW
\end{equation}
with \(W = {\Delta'}^{-1}X^{q^{s}-1}\). From the rule
\begin{equation}
	\prod_{c \in \mathds{F}} (x+cy) = x^{q} - xy^{q-1}
\end{equation}
we get
\begin{equation}
	\Pi(1) = S_{T}^{q} - S_{T}W^{q-1},
\end{equation}
which starts 
\begin{equation} \label{Eq.Pi-1}
	\Pi(1) = 1 - {\Delta'}^{1-q} X^{(q-1)(q^{s}-1)} + \text{higher terms}.
\end{equation}

\subsection{}\label{Subsection.Pi-2}% 
Next, we treat \(\Pi(2)\). For \(a = T^{2} + c_{1}T + c_{0}\),
\begin{equation}
	S_{a}(X) = 1 + m_{1}X^{q^{2s}-q^{2s-1}} + \cdots + m_{s-1}X^{q^{2s}-q} + \frac{a}{\Delta'} X^{q^{2s}-1}
\end{equation}
with the coefficients \(m_{n} = \frac{\prescript{}{a}\ell_{2s-n}'}{\Delta_{(2)}'}\) of \(X^{q^{2s}-n}\), where
\begin{itemize}
	\item \(\Delta_{(2)}' = (\Delta')^{q^{2s}+1}\) is independent of \(a\);
	\item \(m_{0} = 1\), \(m_{1}\), \dots, \(m_{s-1}\) are independent of \(a\);
	\item \(m_{s}, \dots, m_{2s-1}\) depend affine-linearly on \(c_{1}\), but not on \(c_{0}\).
\end{itemize}
Writing \(S_{a}(X) = Z_{0}(c_{1}) + c_{0}W_{0}\) with \(Z_{0}(c_{1}) = S_{T^{2}+c_{1}T}\), \(w_{0} = (\Delta_{(2)}')^{-1}X^{q^{2s}-1}\), we find
for fixed \(c_{1}\):
\begin{equation}
	S(c_{1}) \defeq \prod_{c_{0} \in \mathds{F}} S_{T^{2} + c_{1}T + c_{0}}(X) = Z_{0}(c_{1})^{q} - Z_{0}(c_{1})W_{0}^{q-1}.
\end{equation}
We repeat the argument and write \(S(c_{1}) = Z_{1} + c_{1}W_{1}\) with \(Z_{1}\) and \(W_{1}\) independent of \(c_{1}\). A closer look shows 
\begin{align}
	o(Z_{1} - 1)	&= o(S(c_{1}) - 1) = (q-1) \cdot o(W_{0}) = (q-1)(q^{2s}-1); \\
	o(W_{1})		&= q^{2s+1} - q^{s+1} \nonumber ,	
\end{align}
as only the coefficients \(m_{s}\), \(m_{s+1}\), \dots contribute to \(W_{1}\). Now 
\[
	\Pi(2) = \prod_{c_{1} \in \mathds{F}} S(c_{1}) = Z_{1}^{q} - Z_{1}W_{1}^{q-1},
\]
so
\[
	o(\Pi(2)-1) \geq \min\{ q \cdot o(Z_{1}-1), (q-1) \cdot o(W_{1}) \} = (q-1) \cdot o(W_{1}) = (q-1)q(q^{2s}-q^{s}),
\]
and we have in fact equality since the two arguments of min are different. Hence we have shown:

\begin{Proposition}\label{Proposition.The-product-Pi-2}
	The product \(\Pi(2) = \prod_{a \text{ monic}, \deg a = 2} S_{a}(X)\) satisfies 
	\[
		o(\Pi(2)-1) = (q-1)q(q^{2s}-q^{s}) = (q-1)(q^{2r-1}-q^{r}).
	\]	
\end{Proposition}

\begin{Remark}\label{Remark.Bound-for-Pid-1}
	Obviously, \(o(\Pi(d)-1)\) for \(d \geq 3\) is much larger; each factor \(S_{a}(X)\) satisfies \(o(S_{a}-1) = q^{(r-1)d} - q^{(r-1)d-1}\), and
	already the first step of the procedure \ref{Subsection.Pi-2} gives
	\[
		o\Big( \prod_{c_{0}} S_{a}(X) - 1 \Big) = (q-1)(q^{(r-1)d}-1)
	\]	
	for the product over all \(a = T^{d} + c_{d-1}T^{d-1} + \cdots + c_{1}T + c_{0}\) with \(c_{d-1}, \dots, c_{1}\) fixed. Hence such
	\(\Pi(d)\) don't contribute to the first \(\approx q^{3r-3}\) coefficients of the complete product \(\prod_{a \text{ monic}} S_{a}(X)\).
\end{Remark}

\begin{Theorem}\label{Theorem.h-function-as-power-series-in-t}
	The \(h\)-function as a power series in \(t\) satisfies 
	\[
		h = {h'}^{q} t \Pi(1)^{{-}1} + o(t^{N}),
	\]	
	where \(N = (q-1)(q^{2r-1}-q^{r}) + 1\) and \(\Pi(1)=\Pi(1)(t)\) is the polynomial defined by 
	\begin{align*}
		\Pi(1)(t)	&= S_{T}(t)^{q} - S_{T}(t) W(t)^{q-1}, \\
		S_{T}(t)	&= 1 + \frac{g_{r-2}'}{\Delta'} t^{q^{r-1}-q^{r-2}} + \frac{g_{r-3}'}{\Delta'} t^{q^{r-1}-q^{r-3}} + \cdots + \frac{g_{1}'}{\Delta'} t^{q^{r-1}-q} + \frac{T}{\Delta'} t^{q^{r-1}-1}, \\
		W(t)		&= {\Delta'}^{{-}1} t^{q^{r-1}-1}.	
	\end{align*}
	\(\Pi(1)\) may be inverted as a power series by \(\Pi(1)(t)^{{-}1} = 1 - U + U^{2} - \dots\) with \(U = \Pi(1)(t) - 1\).
\end{Theorem}

\begin{proof}
	We replace the product \(\prod_{a \text{ monic}} S_{a}(t)^{q^{r}-1}\) in \eqref{Eq.Product-h} with 
	\(\prod_{a \text{ monic}, \deg a = 1} S_{a}(t)^{{-}1} = \Pi(1)(t)^{{-}1}\). By Proposition \ref{Proposition.The-product-Pi-2} and 
	Remark \ref{Remark.Bound-for-Pid-1}, the error is \(o(t^{N})\) with \(N = q^{r} \cdot o(\Pi(1) - 1) + 1\) as stated, since 
	\(o(\Pi(d)-1)+1 \geq N\) for \(d \geq 2\).
\end{proof}

We note that by \eqref{Eq.Pi-1} the expansion for \(h\) starts as 
\begin{equation}
	h = {h'}^{q} t \big( 1 + (\Delta')^{1-q}t^{(q^{r-1}-1)(q-1)} + \text{higher terms} \big).
\end{equation}
The reader may verify that Theorem \ref{Theorem.h-function-as-power-series-in-t} is compatible with Theorem \ref{Theorem.gk-as-power-series-in-t}
for \(k=r\), where \(g_{r} = \Delta = ({-}1)^{r-1}h^{q-1}\).


\begin{bibdiv}
	\begin{biblist}
		\bib{Armana11}{article}{
   			AUTHOR = {Armana, C\'ecile},
    		TITLE = {Coefficients of {D}rinfeld modular forms and {H}ecke operators},
  			JOURNAL = {J. Number Theory},
 			FJOURNAL = {Journal of Number Theory},
   			VOLUME = {131},
     		YEAR = {2011},
   			NUMBER = {8},
    		PAGES = {1435--1460},
     		ISSN = {0022-314X,1096-1658},
  			MRCLASS = {11F52 (11F25)},
 			MRNUMBER = {2793886},
			MRREVIEWER = {Mihran\ Papikian},
      		DOI = {10.1016/j.jnt.2011.02.011},
      		URL = {https://doi.org/10.1016/j.jnt.2011.02.011},
		}

		\bib{BandiniValentino23}{article}{
   			AUTHOR = {Bandini, Andrea},
   			AUTHOR = {Valentino, Maria},
    		TITLE = {Fourier coefficients and slopes of {D}rinfeld modular forms},
  			JOURNAL = {Int. J. Number Theory},
 			FJOURNAL = {International Journal of Number Theory},
   			VOLUME = {19},
     		YEAR = {2023},
   			NUMBER = {10},
    		PAGES = {2523--2553},
     		ISSN = {1793-0421,1793-7310},
  			MRCLASS = {11F52 (11F30 11F33)},
 			MRNUMBER = {4672095},
			MRREVIEWER = {Fu-Tsun\ Wei},
      		DOI = {10.1142/s1793042123501245},
      		URL = {https://doi.org/10.1142/s1793042123501245},
		}

		\bib{Basson17}{article}{
   			AUTHOR = {Basson, Dirk},
    		TITLE = {A product formula for the higher rank {D}rinfeld discriminant function},
  			JOURNAL = {J. Number Theory},
 			FJOURNAL = {Journal of Number Theory},
  	 		VOLUME = {178},
     		YEAR = {2017},
    		PAGES = {190--200},
     		ISSN = {0022-314X,1096-1658},
  			MRCLASS = {11F52},
 			MRNUMBER = {3646835},
			MRREVIEWER = {Ahmad\ El-Guindy},
     	 	DOI = {10.1016/j.jnt.2017.02.010},
      		URL = {https://doi.org/10.1016/j.jnt.2017.02.010},
		}

		\bib{BassonBreuerPink24}{article}{
   			AUTHOR = {Basson, Dirk},
   			AUTHOR = {Breuer, Florian},
   			AUTHOR = {Pink, Richard},
    		TITLE = {Drinfeld modular forms of arbitrary rank},
 	 		JOURNAL = {Mem. Amer. Math. Soc.},
 			FJOURNAL = {Memoirs of the American Mathematical Society},
   			VOLUME = {304},
     		YEAR = {2024},
   			NUMBER = {1531},
    		PAGES = {viii+79},
     		ISSN = {0065-9266,1947-6221},
     		ISBN = {978-1-4704-7223-8; 978-1-4704-8010-3},
  			MRCLASS = {11F52 (11G09)},
 			MRNUMBER = {4850411},
      		DOI = {10.1090/memo/1531},
      		URL = {https://doi.org/10.1090/memo/1531},
		}

		\bib{Bosser02}{article}{
   			AUTHOR = {Bosser, Vincent},
    		TITLE = {Congruence properties of the coefficients of the {D}rinfeld modular invariant},
  			JOURNAL = {Manuscripta Math.},
 			FJOURNAL = {Manuscripta Mathematica},
   			VOLUME = {109},
     		YEAR = {2002},
   			NUMBER = {3},
    		PAGES = {289--307},
     		ISSN = {0025-2611,1432-1785},
  			MRCLASS = {11F52 (11F33 11G09)},
 			MRNUMBER = {1948016},
			MRREVIEWER = {Ignazio\ Longhi},
      		DOI = {10.1007/s00229-002-0299-3},
     	 	URL = {https://doi.org/10.1007/s00229-002-0299-3},
		}

		\bib{Choi08}{article}{
   			AUTHOR = {Choi, SoYoung},
    		TITLE = {Linear relations and congruences for the coefficients of {D}rinfeld modular forms},
  			JOURNAL = {Israel J. Math.},
 			FJOURNAL = {Israel Journal of Mathematics},
   			VOLUME = {165},
     		YEAR = {2008},
    		PAGES = {93--101},
     		ISSN = {0021-2172,1565-8511},
  			MRCLASS = {11F52},
 			MRNUMBER = {2403616},
			MRREVIEWER = {Ignazio\ Longhi},
      		DOI = {10.1007/s11856-008-1005-2},
      		URL = {https://doi.org/10.1007/s11856-008-1005-2},
		}

		\bib{ChoiIm14}{article}{
   			AUTHOR = {Choi, SoYoung},
   			AUTHOR = {Im, Bo-Hae},
    		TITLE = {The zeros of certain weakly holomorphic {D}rinfeld modular forms},
  			JOURNAL = {Manuscripta Math.},
 			FJOURNAL = {Manuscripta Mathematica},
   			VOLUME = {144},
     		YEAR = {2014},
   			NUMBER = {3-4},
    		PAGES = {503--515},
     		ISSN = {0025-2611,1432-1785},
  			MRCLASS = {11F03 (11F37)},
 			MRNUMBER = {3227524},
      		DOI = {10.1007/s00229-014-0660-3},
      		URL = {https://doi.org/10.1007/s00229-014-0660-3},
		}

		\bib{DalalKumar23}{article}{
   			AUTHOR = {Dalal, Tarun},
   			AUTHOR = {Kumar, Narasimha},
    		TITLE = {The structure of {D}rinfeld modular forms of level {$\Gamma_0(T)$} and applications},
  			JOURNAL = {J. Algebra},
 			FJOURNAL = {Journal of Algebra},
  	 		VOLUME = {619},
     		YEAR = {2023},
    		PAGES = {778--798},
     		ISSN = {0021-8693,1090-266X},
  			MRCLASS = {11F52 (11F33 11G09 13C05)},
 			MRNUMBER = {4531537},
			MRREVIEWER = {Fu-Tsun\ Wei},
      		DOI = {10.1016/j.jalgebra.2022.11.027},
      		URL = {https://doi.org/10.1016/j.jalgebra.2022.11.027},
		}

		\bib{Gekeler85}{article}{
   			AUTHOR = {Gekeler, Ernst-Ulrich},
    		TITLE = {A product expansion for the discriminant function of {D}rinfel\cprime d{} modules of rank two},
  			JOURNAL = {J. Number Theory},
 			FJOURNAL = {Journal of Number Theory},
   			VOLUME = {21},
     		YEAR = {1985},
   			NUMBER = {2},
    		PAGES = {135--140},
     		ISSN = {0022-314X,1096-1658},
  			MRCLASS = {11F99 (11R58)},
 			MRNUMBER = {808282},
			MRREVIEWER = {David\ Goss},
      		DOI = {10.1016/0022-314X(85)90046-0},
      		URL = {https://doi.org/10.1016/0022-314X(85)90046-0},
		}

		\bib{Gekeler88}{article}{
   			AUTHOR = {Gekeler, Ernst-Ulrich},
    		TITLE = {On the coefficients of {D}rinfel\cprime d{} modular forms},
  			JOURNAL = {Invent. Math.},
 			FJOURNAL = {Inventiones Mathematicae},
   			VOLUME = {93},
     		YEAR = {1988},
   			NUMBER = {3},
    		PAGES = {667--700},
     		ISSN = {0020-9910,1432-1297},
  			MRCLASS = {11F85 (11G20 11R58)},
 			MRNUMBER = {952287},
			MRREVIEWER = {David\ Goss},
      		DOI = {10.1007/BF01410204},
      		URL = {https://doi.org/10.1007/BF01410204},
		}

		\bib{Gekeler99}{article}{
   			AUTHOR = {Gekeler, Ernst-Ulrich},
    		TITLE = {Growth order and congruences of coefficients of the {D}rinfeld discriminant function},
  			JOURNAL = {J. Number Theory},
 			FJOURNAL = {Journal of Number Theory},
   			VOLUME = {77},
     		YEAR = {1999},
   			NUMBER = {2},
    		PAGES = {314--325},
     		ISSN = {0022-314X,1096-1658},
  			MRCLASS = {11F52 (11F33)},
 			MRNUMBER = {1702216},
			MRREVIEWER = {Yoshinori\ Hamahata},
      		DOI = {10.1006/jnth.1999.2387},
      		URL = {https://doi.org/10.1006/jnth.1999.2387},
		}

		\bib{Gekeler17}{article}{
   			AUTHOR = {Gekeler, Ernst-Ulrich},
    		TITLE = {On {D}rinfeld modular forms of higher rank},
  			JOURNAL = {J. Th\'eor. Nombres Bordeaux},
 			FJOURNAL = {Journal de Th\'eorie des Nombres de Bordeaux},
  		 	VOLUME = {29},
     		YEAR = {2017},
   			NUMBER = {3},
    		PAGES = {875--902},
     		ISSN = {1246-7405,2118-8572},
  			MRCLASS = {11F52 (11G09 14G22)},
 			MRNUMBER = {3745253},
			MRREVIEWER = {SoYoung\ Choi},
      		DOI = {10.5802/jtnb.1005},
      		URL = {https://doi.org/10.5802/jtnb.1005},
		}

		\bib{Gekeler22}{article}{
   			AUTHOR = {Gekeler, Ernst-Ulrich},
    		TITLE = {On {D}rinfeld modular forms of higher rank {II}},
  			JOURNAL = {J. Number Theory},
 			FJOURNAL = {Journal of Number Theory},
   			VOLUME = {232},
     		YEAR = {2022},
    		PAGES = {4--32},
     		ISSN = {0022-314X,1096-1658},
  			MRCLASS = {11F52 (11G09)},
 			MRNUMBER = {4343822},
			MRREVIEWER = {Ahmad\ El-Guindy},
      		DOI = {10.1016/j.jnt.2018.11.011},
      		URL = {https://doi.org/10.1016/j.jnt.2018.11.011},
		}

		\bib{Gekeler22-2}{article}{
  	 		AUTHOR = {Gekeler, Ernst-Ulrich},
    		TITLE = {On {D}rinfeld modular forms of higher rank {IV}: {M}odular forms with level},
  			JOURNAL = {J. Number Theory},
 			FJOURNAL = {Journal of Number Theory},
   			VOLUME = {232},
     		YEAR = {2022},
    		PAGES = {33--74},
     		ISSN = {0022-314X,1096-1658},
  			MRCLASS = {11F52 (11G09 14G35)},
 			MRNUMBER = {4343823},
			MRREVIEWER = {Maria\ Valentino},
     	 	DOI = {10.1016/j.jnt.2019.04.019},
      		URL = {https://doi.org/10.1016/j.jnt.2019.04.019},
		}

		\bib{Gekeler22-3}{article}{
   			AUTHOR = {Gekeler, Ernst-Ulrich},
    		TITLE = {Goss polynomials, {$q$}-adic expansions, and {S}heats compositions},
  			JOURNAL = {J. Number Theory},
 			FJOURNAL = {Journal of Number Theory},
   			VOLUME = {240},
     		YEAR = {2022},
    		PAGES = {196--253},
     		ISSN = {0022-314X,1096-1658},
  			MRCLASS = {11T55 (11B65 11F52 11M38 33E50)},
 			MRNUMBER = {4458239},
			MRREVIEWER = {Maria\ Valentino},
      		DOI = {10.1016/j.jnt.2022.01.008},
      		URL = {https://doi.org/10.1016/j.jnt.2022.01.008},
      		NOTE = {see also Corrigendum J. Number Theory 246 (2023) 328-331},
		}


		\bib{Gekeler25}{article}{
   			AUTHOR = {Gekeler, Ernst-Ulrich},
    		TITLE = {On {D}rinfeld modular forms of higher rank {VII}: {E}xpansions at the boundary},
  			JOURNAL = {J. Number Theory},
 			FJOURNAL = {Journal of Number Theory},
   			VOLUME = {269},
     		YEAR = {2025},
    		PAGES = {260--340},
     		ISSN = {0022-314X,1096-1658},
  			MRCLASS = {11F52 (11G16 11G18 14D22 14G22)},
 			MRNUMBER = {4833117},
			MRREVIEWER = {Fu-Tsun\ Wei},
      		DOI = {10.1016/j.jnt.2024.09.015},
      		URL = {https://doi.org/10.1016/j.jnt.2024.09.015},
		}
		
		\bib{Gekeler25-2}{misc}{
			AUTHOR = {Gekeler, Ernst-Ulrich},
			TITLE = {Modular forms for \(\GL(r,\mathds{F}_{q}[T]\): Hecke operators and growth of expansion coefficients},
			NOTE = {arXiv:2511.01712},
			YEAR = {2025},
		}

		\bib{Lopez10}{article}{
   			AUTHOR = {L\'opez, Bartolom\'e},
    		TITLE = {A non-standard {F}ourier expansion for the {D}rinfeld discriminant function},
  			JOURNAL = {Arch. Math. (Basel)},
 			FJOURNAL = {Archiv der Mathematik},
   			VOLUME = {95},
     		YEAR = {2010},
   			NUMBER = {2},
    		PAGES = {143--150},
     		ISSN = {0003-889X,1420-8938},
  			MRCLASS = {11F52 (11G09)},
 			MRNUMBER = {2674250},
			MRREVIEWER = {SoYoung\ Choi},
      		DOI = {10.1007/s00013-010-0148-7},
      		URL = {https://doi.org/10.1007/s00013-010-0148-7},
		}

		\bib{Papikian23}{book}{
   			AUTHOR = {Papikian, Mihran},
    		TITLE = {Drinfeld modules},
   			SERIES = {Graduate Texts in Mathematics},
   			VOLUME = {296},
			PUBLISHER = {Springer, Cham},
     		YEAR = {2023},
    		PAGES = {xxi+526},
     		ISBN = {978-3-031-19706-2; 978-3-031-19707-9},
  			MRCLASS = {11G09 (11-02 11R58)},
 			MRNUMBER = {4592575},
			MRREVIEWER = {Ernst-Ulrich\ Gekeler},
      		DOI = {10.1007/978-3-031-19707-9},
      		URL = {https://doi.org/10.1007/978-3-031-19707-9},
		}

		\bib{Sugiyama18}{article}{
   			AUTHOR = {Sugiyama, Yusuke},
    		TITLE = {The integrality and reduction of {D}rinfeld modular forms of arbitrary rank},
  			JOURNAL = {J. Number Theory},
 			FJOURNAL = {Journal of Number Theory},
   			VOLUME = {188},
     		YEAR = {2018},
    		PAGES = {371--391},
     		ISSN = {0022-314X,1096-1658},
  			MRCLASS = {11F52},
 			MRNUMBER = {3778640},
			MRREVIEWER = {SoYoung\ Choi},
      		DOI = {10.1016/j.jnt.2018.01.017},
     	 	URL = {https://doi.org/10.1016/j.jnt.2018.01.017},
		}

		\bib{Vicent10}{article}{
   			AUTHOR = {Vincent, Christelle},
    		TITLE = {Drinfeld modular forms modulo {$\germ p$}},
  			JOURNAL = {Proc. Amer. Math. Soc.},
 			FJOURNAL = {Proceedings of the American Mathematical Society},
   			VOLUME = {138},
     		YEAR = {2010},
   			NUMBER = {12},
    		PAGES = {4217--4229},
     		ISSN = {0002-9939,1088-6826},
  			MRCLASS = {11F52 (11F25 11F30 11F33)},
 			MRNUMBER = {2680048},
			MRREVIEWER = {Mihran\ Papikian},
      		DOI = {10.1090/S0002-9939-2010-10459-8},
      		URL = {https://doi.org/10.1090/S0002-9939-2010-10459-8},
		}
	\end{biblist}
\end{bibdiv}
\end{document}